\documentclass{amsart}
\usepackage{tikz}
\usepackage{hyperref,graphicx}
\usepackage[backend=biber]{biblatex}
\newcommand{\Dfn}[1]{\emph{#1}}%
\newtheorem{thm}{Theorem}[section]

\newtheorem{lem}[thm]{Lemma}

\title[Alternating sign matrices through X-rays]{The set of
  alternating sign matrices which are determined by their X-ray is a
  member of the Catalan family}%
\author{Martin Rubey}%
\email{\url{Martin.Rubey@tuwien.ac.at}}%
\address{Fakult\"at f\"ur Mathematik und Geoinformation, TU Wien,
  Austria}%

\begin{document}
\begin{abstract}
  We exhibit a bijection between Dyck paths and alternating sign
  matrices which are determined by their antidiagonal sums.
\end{abstract}
\maketitle

\section{Introduction}
A fundamental question in discrete tomography is whether a binary
image can be reconstructed from a small number of projections.  As a
special case, one might restrict attention to permutation matrices,
and try to determine which vectors of antidiagonal sums appear only
once.  This problem, considered by Bebeacua, Mansour, Postnikov and
Severini~\cite{MR2301096}, is apparently still open.

In this note, we consider the analogous problem for alternating sign
matrices.  An alternating sign matrix is a square matrix of $0$s,
$1$s and $-1$s such that the sum of each row and each column is $1$,
and the nonzero entries in each row and in each column alternate in
sign.  For an $n\times n$-alternating sign matrix $A$, the $k$-th
(antidiagonal) sum is $x_k = \sum_{i+j=k+1} A_{i,j}$ and the
(antidiagonal) \Dfn{X-ray} is the vector $x_1,\dots,x_{2n-1}$.  For
example, the alternating sign matrices of size three together with
their X-rays are as follows:
\begin{center}
\scalebox{0.8}{\mbox{%
$\begin{array}{ccccccc}
\begin{pmatrix}
1 & 0 & 0 \\
0 & 1 & 0 \\
0 & 0 & 1
\end{pmatrix}& \begin{pmatrix}
0 & 1 & 0 \\
1 & 0 & 0 \\
0 & 0 & 1
\end{pmatrix}& \begin{pmatrix}
1 & 0 & 0 \\
0 & 0 & 1 \\
0 & 1 & 0
\end{pmatrix}& \begin{pmatrix}
0 & 1 & 0 \\
1 & -1 & 1 \\
0 & 1 & 0
\end{pmatrix}& \begin{pmatrix}
0 & 0 & 1 \\
0 & 1 & 0 \\
1 & 0 & 0
\end{pmatrix}& \begin{pmatrix}
0 & 0 & 1 \\
1 & 0 & 0 \\
0 & 1 & 0
\end{pmatrix}& \begin{pmatrix}
0 & 1 & 0 \\
0 & 0 & 1 \\
1 & 0 & 0
\end{pmatrix}\\[16pt]
1/ 0/ 1/ 0/ 1&
0/ 2/ 0/ 0/ 1&
1/ 0/ 0/ 2/ 0&
0/ 2/-1/ 2/ 0&
0/ 0/ 3/ 0/ 0&
0/ 1/ 1/ 1/ 0&
0/ 1/ 1/ 1/ 0
\end{array}$
}}
\end{center}
Note that all X-rays except $0/ 1/ 1/ 1/ 0$ occur precisely once.
Thus, there are five alternating sign matrices determined by there
X-rays.  We can now state our main result:
\begin{thm}
  There is an explicit bijection between Dyck paths of semilength $n$
  and $n\times n$-alternating sign matrices which are determined by
  their X-rays.
\end{thm}

The coincidence described by the theorem was observed when submitting
the statistic\footnote{\url{http://www.findstat.org/St000889}}
counting the number of alternating sign matrices with the same X-rays
to the online database of combinatorial statistics
FindStat~\cite{FindStat2017} and looking at the first few generating
functions automatically produced there.  We currently have no
explanation for any of the other terms in the distribution.

\section{The bijection}
The map $\mathcal A$ from Dyck paths to alternating sign matrices is
defined as follows, see Figure~\ref{fig:exampleA} for an example.
For more visual clarity in the pictures, we use ($+$)s and ($-$)s
instead of $1$s and $-1$s and omit $0$s.
\begin{itemize}
\item Draw the Dyck path in an $n\times n$ square, beginning in the
  top left corner, taking east and south steps and terminating in the
  bottom right corner, never going below the main diagonal of the
  matrix.
\item Add the reflection through the main diagonal of the Dyck path
  to the picture.
\item For each peak of the Dyck path, fill the cells lying between
  the peak and its mirror image on the antidiagonal with $1$s.
\item For each valley of the Dyck path, fill the cells lying between
  the valley and its mirror image on the antidiagonal with $-1$s.
\item Fill the remaining cells with $0$s.
\end{itemize}

\begin{figure}[h]
  \centering
  \begin{tikzpicture}[scale=0.5]
    \draw [line width=0.5] (0,0) -- (0,8) -- (8,8) -- (8,0) -- cycle;
    \draw[dotted] (0,8)--(8,0);
    \draw[rounded corners=1, line width=1]%
    (0,8)--(4,8)--%
    (4,7)--(6,7)--%
    (6,5)--(7,5)--%
    (7,4)--(8,4)--%
    (8,0);%
    \draw[rounded corners=1, line width=1, dotted]%
    (0,8)--(0,4)--%
    (1,4)--(1,2)--%
    (3,2)--(3,1)--%
    (4,1)--(4,0)--%
    (8,0);%
    \node[below left] at (4,8) {$1$};
    \node[below left] at (3,7) {$1$};
    \node[below left] at (2,6) {$1$};
    \node[below left] at (1,5) {$1$};
    \node[below left] at (6,7) {$1$};
    \node[below left] at (5,6) {$1$};
    \node[below left] at (4,5) {$1$};
    \node[below left] at (3,4) {$1$};
    \node[below left] at (2,3) {$1$};
    \node[below left] at (7,5) {$1$};
    \node[below left] at (6,4) {$1$};
    \node[below left] at (5,3) {$1$};
    \node[below left] at (4,2) {$1$};
    \node[below left] at (8,4) {$1$};
    \node[below left] at (7,3) {$1$};
    \node[below left] at (6,2) {$1$};
    \node[below left] at (5,1) {$1$};
    \node[below left] at (4,7) {$-1$};
    \node[below left] at (3,6) {$-1$};
    \node[below left] at (2,5) {$-1$};
    \node[below left] at (6,5) {$-1$};
    \node[below left] at (5,4) {$-1$};
    \node[below left] at (4,3) {$-1$};
    \node[below left] at (7,4) {$-1$};
    \node[below left] at (6,3) {$-1$};
    \node[below left] at (5,2) {$-1$};
    \foreach\x in {1,2,3,5,6,7,8}
      \node[below left] at (\x,8) {$0$};
    \foreach\x in {1,2,5,7,8}
      \node[below left] at (\x,7) {$0$};
    \foreach\x in {1,4,6,7,8}
      \node[below left] at (\x,6) {$0$};
    \foreach\x in {3,5,8}
      \node[below left] at (\x,5) {$0$};
    \foreach\x in {1,2,4}
      \node[below left] at (\x,4) {$0$};
    \foreach\x in {1,3,8}
      \node[below left] at (\x,3) {$0$};
    \foreach\x in {1,2,3,7,8}
      \node[below left] at (\x,2) {$0$};
    \foreach\x in {1,2,3,4,6,7,8}
      \node[below left] at (\x,1) {$0$};
  \end{tikzpicture}
  \quad
  \begin{tikzpicture}[scale=0.5]
    \draw [line width=0.5] (0,0) -- (0,8) -- (8,8) -- (8,0) -- cycle;
    \draw[dotted] (0,8)--(8,0);
    \draw[rounded corners=1, line width=1]%
    (0,8)--(4,8)--%
    (4,7)--(6,7)--%
    (6,5)--(7,5)--%
    (7,4)--(8,4)--%
    (8,0);%
    \draw[rounded corners=1, line width=1, dotted]%
    (0,8)--(0,4)--%
    (1,4)--(1,2)--%
    (3,2)--(3,1)--%
    (4,1)--(4,0)--%
    (8,0);%
    \node[below left] at (4,8) {$+$};
    \node[below left] at (3,7) {$+$};
    \node[below left] at (2,6) {$+$};
    \node[below left] at (1,5) {$+$};
    \node[below left] at (6,7) {$+$};
    \node[below left] at (5,6) {$+$};
    \node[below left] at (4,5) {$+$};
    \node[below left] at (3,4) {$+$};
    \node[below left] at (2,3) {$+$};
    \node[below left] at (7,5) {$+$};
    \node[below left] at (6,4) {$+$};
    \node[below left] at (5,3) {$+$};
    \node[below left] at (4,2) {$+$};
    \node[below left] at (8,4) {$+$};
    \node[below left] at (7,3) {$+$};
    \node[below left] at (6,2) {$+$};
    \node[below left] at (5,1) {$+$};
    \node[below left] at (4,7) {$-$};
    \node[below left] at (3,6) {$-$};
    \node[below left] at (2,5) {$-$};
    \node[below left] at (6,5) {$-$};
    \node[below left] at (5,4) {$-$};
    \node[below left] at (4,3) {$-$};
    \node[below left] at (7,4) {$-$};
    \node[below left] at (6,3) {$-$};
    \node[below left] at (5,2) {$-$};
  \end{tikzpicture}
  \caption{The image of A Dyck path}
  \label{fig:exampleA}
\end{figure}
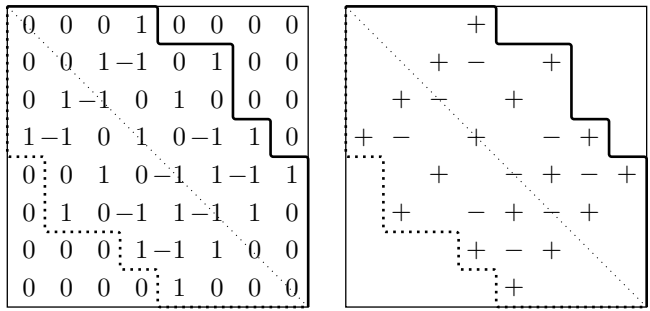

\section{A map on diagonally symmetric alternating sign matrices}
Because transposing a matrix preserves the X-ray, only diagonally
symmetric alternating sign matrices may be reconstructible from their
X-ray.  We now present a map $\mathcal M$ on diagonally symmetric
alternating sign matrices that preserves the X-ray and is the
identity precisely on the matrices in the image of the map
$\mathcal A$ from the previous section.  Let $A$ be a diagonally
symmetric alternating sign matrix, then $\mathcal M(A)$ is obtained
as follows, see Figure~\ref{fig:exampleM} for an example.
\begin{itemize}
\item Imagine a sun in the north-east, such that the $1$s in $A$ cast
  shadows, and trace out a Dyck path by following the shadow line.
\item Reflect the entries of $A$ which are strictly south-west of the
  entries just below the Dyck path through the subdiagonal.
\item Into each cell just south-west of a valley of the Dyck path
  which is not on the subdiagonal and which contains a $0$, place a
  $-1$, and place a $1$ in the cell reflected through the
  subdiagonal.
\end{itemize}
Note that the Dyck path constructed in the first step returns to the
main diagonal exactly once for each direct summand of $A$, regarding
$A$ as a block diagonal matrix.  Thus, the map $\mathcal M$ is such
that it can be applied to each direct summand of $A$ individually.



\begin{figure}[h]
  \centering
  \begin{tikzpicture}[scale=0.40]
    \draw [line width=0.5] (0,0) -- (0,8) -- (8,8) -- (8,0) -- cycle;
    \draw [fill=gray, opacity=0.3]
    (0,7) -- (3,7) -- (3,6) -- (5,6) -- (5,4) -- (6,4) -- (6,3) --
    (7,3) -- (7,0) -- (4,0) -- (4,1) -- (3,1) -- (3,2) -- (1,2) --
    (1,4) -- (0,4) -- cycle;
    \draw[dotted] (0,7)--(7,0);
    \draw[rounded corners=1, line width=1]%
    (0,8)--(4,8)--%
    (4,7)--(6,7)--%
    (6,5)--(7,5)--%
    (7,4)--(8,4)--%
    (8,0);%
    \draw[rounded corners=1, line width=1, dotted]%
    (0,8)--(0,4)--%
    (1,4)--(1,2)--%
    (3,2)--(3,1)--%
    (4,1)--(4,0)--%
    (8,0);%
    \node[below left] at (2,7) {$+$};
    \node[below left] at (4,8) {$+$};
    \node[below left] at (1,5) {$+$};
    \node[below left] at (6,7) {$+$};
    \node[below left] at (5,6) {$+$};
    \node[below left] at (3,4) {$+$};
    \node[below left] at (2,3) {$+$};
    \node[below left] at (7,5) {$+$};
    \node[below left] at (4,2) {$+$};
    \node[below left] at (8,4) {$+$};
    \node[below left] at (7,3) {$+$};
    \node[below left] at (6,2) {$+$};
    \node[below left] at (5,1) {$+$};
    \node[below left] at (4,7) {$-$};
    \node[below left] at (2,5) {$-$};
    \node[below left] at (7,4) {$-$};
    \node[below left] at (6,3) {$-$};
    \node[below left] at (5,2) {$-$};
  \end{tikzpicture}
  \quad\raisebox{50pt}{$\mapsto$}\quad
  \begin{tikzpicture}[scale=0.40]
    \draw [line width=0.5] (0,0) -- (0,8) -- (8,8) -- (8,0) -- cycle;
    \draw [fill=gray, opacity=0.3]
    (0,7) -- (3,7) -- (3,6) -- (5,6) -- (5,4) -- (6,4) -- (6,3) --
    (7,3) -- (7,0) -- (4,0) -- (4,1) -- (3,1) -- (3,2) -- (1,2) --
    (1,4) -- (0,4) -- (0,7);
    \draw[dotted] (0,7)--(7,0);
    \draw[rounded corners=1, line width=1]%
    (0,8)--(4,8)--%
    (4,7)--(6,7)--%
    (6,5)--(7,5)--%
    (7,4)--(8,4)--%
    (8,0);%
    \draw[rounded corners=1, line width=1, dotted]%
    (0,8)--(0,4)--%
    (1,4)--(1,2)--%
    (3,2)--(3,1)--%
    (4,1)--(4,0)--%
    (8,0);%
    \node[below left] at (1,6) {$+$};
    \node[below left] at (4,8) {$+$};
    \node[below left] at (3,7) {$+$};
    \node[below left] at (6,7) {$+$};
    \node[below left] at (5,6) {$+$};
    \node[below left] at (4,5) {$+$};
    \node[below left] at (2,3) {$+$};
    \node[below left] at (7,5) {$+$};
    \node[below left] at (6,4) {$+$};
    \node[below left] at (8,4) {$+$};
    \node[below left] at (7,3) {$+$};
    \node[below left] at (6,2) {$+$};
    \node[below left] at (5,1) {$+$};
    \node[below left] at (4,7) {$-$};
    \node[below left] at (3,6) {$-$};
    \node[below left] at (7,4) {$-$};
    \node[below left] at (6,3) {$-$};
    \node[below left] at (5,2) {$-$};
  \end{tikzpicture}
  \quad\raisebox{50pt}{$\mapsto$}\quad
  \begin{tikzpicture}[scale=0.40]
    \draw [line width=0.5] (0,0) -- (0,8) -- (8,8) -- (8,0) -- cycle;
    \draw[dotted] (0,7)--(7,0);
    \draw[rounded corners=1, line width=1]%
    (0,8)--(4,8)--%
    (4,7)--(6,7)--%
    (6,5)--(7,5)--%
    (7,4)--(8,4)--%
    (8,0);%
    \draw[rounded corners=1, line width=1, dotted]%
    (0,8)--(0,4)--%
    (1,4)--(1,2)--%
    (3,2)--(3,1)--%
    (4,1)--(4,0)--%
    (8,0);%
    \node[below left] at (1,6) {$+$};
    \node[below left, red] (A) at (6,5) {$-$};
    \node[draw=black, circle, minimum size=12pt] at (A) {};
    \node[below left, red] (B) at (3,2) {$+$};
    \node[draw=black, circle, minimum size=12pt] at (B) {};
    \node[below left] at (4,8) {$+$};
    \node[below left] at (3,7) {$+$};
    \node[below left] at (6,7) {$+$};
    \node[below left] at (5,6) {$+$};
    \node[below left] at (4,5) {$+$};
    \node[below left] at (2,3) {$+$};
    \node[below left] at (7,5) {$+$};
    \node[below left] at (6,4) {$+$};
    \node[below left] at (8,4) {$+$};
    \node[below left] at (7,3) {$+$};
    \node[below left] at (6,2) {$+$};
    \node[below left] at (5,1) {$+$};
    \node[below left] at (4,7) {$-$};
    \node[below left] at (3,6) {$-$};
    \node[below left] at (7,4) {$-$};
    \node[below left] at (6,3) {$-$};
    \node[below left] at (5,2) {$-$};
  \end{tikzpicture}
  \caption{A diagonally symmetric alternating sign matrix and its
    image}
  \label{fig:exampleM}
\end{figure}
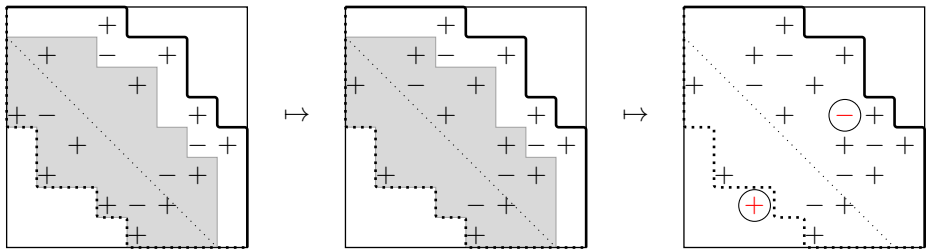

\begin{lem}
  The map $\mathcal M$, applied to a diagonally symmetric alternating
  sign matrix, produces an alternating sign matrix.
\end{lem}
\begin{proof}
  Let $A$ be a diagonally symmetric alternating sign matrix.  Let us
  call the region in $A$ symmetric with respect to the subdiagonal,
  whose south-west border is the reflected Dyck path, the \Dfn{shade}
  of $A$.  This is the shaded region in Figure~\ref{fig:exampleM}.

  Consider a column $c$ of $A$, and its reflection $r$ through the
  subdiagonal.  Thus, when $c$ is the first column, $r$ is the second
  row of $A$.

  Suppose first that the top most nonzero entry of $c$ and the right
  most nonzero entry of $r$ within the shade of $A$ are both $1$.
  This is the case when $c$ (or $r$) does not contain a peak of the Dyck
  path, or the valley below (or to the left of) the peak contains a $-1$.

  In this case, column $c$ and row $r$ of $\mathcal M(A)$ satisfy the
  alternating sign matrix conditions, because after reflecting
  through the subdiagonal the top most nonzero entry of all columns,
  and the right most nonzero entry of all rows within the shade is
  $1$.

  Let us now consider the second scenario, where the top most nonzero
  entry of $c$ within the shade of the original matrix $A$ is $-1$.
  In the example of Figure~\ref{fig:exampleM}, this happens in the
  sixth column.

  In this case, the Dyck path must have a peak in this column.  Let
  $v$ be the cell just south-west of the valley below the peak.  Note
  that $v$ must contain a $0$.

  The cell $v$ must be strictly above the diagonal, because otherwise
  the reflection of the $-1$ below it through the main diagonal would
  lie on or above the Dyck path.  Thus, by definition of
  $\mathcal M$, we place a $-1$ into the cell $v$.  The effect of
  this is that column $c$ of $\mathcal M(A)$ is alternating.

  Furthermore, we place a $1$ in the cell $v'$ corresponding to $v$
  reflected through the subdiagonal.  This satisfies the alternating
  sign matrix conditions, because after reflecting through the
  subdiagonal, the row containing $v'$ begins with a $-1$.
\end{proof}

\begin{lem}
  The map $\mathcal M$, applied to a diagonally symmetric alternating
  sign matrix $A$, is the identity if and only if $A$ is in the image
  of $\mathcal A$.
\end{lem}
\begin{proof}
  If $A$ is in the image of $\mathcal A$, the shade of $A$ is
  symmetric.  Moreover, the cells just south-west of the valleys
  which are above the subdiagonal all contain $-1$s.  Thus,
  $\mathcal M(A) = A$.

  Otherwise, since $A$ is symmetric, and the shade of $A$ is
  reflected through the subdiagonal, $\mathcal M(A)$ cannot be
  symmetric.
\end{proof}

\section{Reconstructing the alternating sign matrix}

To complete the proof of the main theorem, we have to show the following:

\begin{lem}
  The X-ray corresponding to an alternating sign matrix in the image
  of $\mathcal A$ determines the matrix unambiguously.
\end{lem}
\begin{proof}
  Consider the antidiagonal sums beginning at the north-west corner.
  Suppose that the entries of the first $k$ antidiagonals are
  uniquely determined by their X-rays $x_1/x_2/\dots/x_k$, and
  suppose that $x_k\neq 0$, $x_{k+1}=\dots=x_{\ell-1}=0$ and
  $x_{\ell}\neq 0$.

  For simplicity, assume that $x_k > 0$.  By hypothesis, the
  alternating sign matrix then has the following form:
  
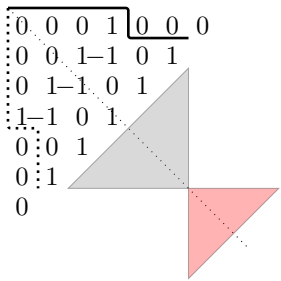
\begin{center}
    \begin{tikzpicture}[scale=0.4]
      \draw[dotted] (0,8)--(8,0);
      \draw[rounded corners=1, line width=1]%
      (0,8)--(4,8)--%
      (4,7)--(6,7);%
      \draw[rounded corners=1, line width=1, dotted]%
      (0,8)--(0,4)--%
      (1,4)--(1,2);%
      \draw [fill=gray, opacity=0.3] (6,6)--(6,2)--(2,2)--cycle;
      \draw [fill=red, opacity=0.3] (6,2)--(9,2)--(6,-1)--cycle;
      \node[below left] at (4,8) {$1$};
      \node[below left] at (3,7) {$1$};
      \node[below left] at (2,6) {$1$};
      \node[below left] at (1,5) {$1$};
      \node[below left] at (4,7) {$-1$};
      \node[below left] at (3,6) {$-1$};
      \node[below left] at (2,5) {$-1$};
      \node[below left] at (6,7) {$1$};
      \node[below left] at (5,6) {$1$};
      \node[below left] at (4,5) {$1$};
      \node[below left] at (3,4) {$1$};
      \node[below left] at (2,3) {$1$};
      \foreach\x in {1,2,3,5,6,7}
        \node[below left] at (\x,8) {$0$};
      \foreach\x in {1,2,5}
        \node[below left] at (\x,7) {$0$};
      \foreach\x in {1,4}
        \node[below left] at (\x,6) {$0$};
      \foreach\x in {3}
        \node[below left] at (\x,5) {$0$};
      \foreach\x in {1,2}
        \node[below left] at (\x,4) {$0$};
      \foreach\x in {1}
        \node[below left] at (\x,3) {$0$};
      \foreach\x in {1}
        \node[below left] at (\x,2) {$0$};
    \end{tikzpicture}
  \end{center}

  Let us first note that there cannot be any nonzero entries on the
  antidiagonals $k+1,\dots,\ell-1$, since all these have sum zero.
  More precisely, suppose for the sake of contradiction that there is
  such an antidiagonal and consider the first of these.  Because
  every row and every column of an alternating sign matrix must begin
  with a $1$, this antidiagonal can have $-1$s only in the triangular
  region south-east of the sequence of $1$s in the $k$-th
  antidiagonal - shaded grey in the example above.  However, there
  cannot be any $1$s on the same antidiagonal, necessarily outside of
  this triangular region: any such $1$ below the main diagonal would
  be followed by another $1$ in the same column above it.

  We now distinguish two cases: if $x_\ell < 0$, by hypothesis
  $x_\ell$ is so large that all cells of the antidiagonal within the
  triangular region defined above are filled with $-1$s.  Thus, in
  this case the entries on the $\ell$-th antidiagonal are also
  uniquely determined by the antidiagonal sum.

  On the other hand, if $x_\ell > 0$, by hypothesis $x_\ell$ is so
  large that all cells of the antidiagonal that lie within the
  triangular region shaded red in the example above are filled with
  $1$s.  Since there cannot be any $1$s on the same antidiagonal
  outside of the red triangular region, also in this case the entries
  of the $\ell$-th antidiagonal are uniquely determined by their sum.
\end{proof}

\printbibliography
\end{document}